\documentclass[12pt]{article}

\usepackage{amsmath, amssymb, amsfonts, amsthm,authblk}
\usepackage{hyperref}

\newcommand{\R}{\mathbb{R}}
\newcommand{\Z}{\mathbb{Z}}

\newcommand{\Q}{\mathbb{Q}}
\newcommand{\Zp}{\mathbb{Z}_p}

\usepackage{amsmath}

\usepackage{tikz}
\usetikzlibrary{calc}
\usepackage{xcolor}

\numberwithin{equation}{section}
\usepackage{amsfonts}
\usepackage{amsthm}
\usepackage{amssymb}
\usepackage{bbm}
\usepackage{centernot}
\usepackage[x11names]{xcolor}
\usepackage{chemfig}
\usepackage{amsmath}
\usepackage{cite}
\usepackage[nodayofweek]{datetime} 
\usepackage{dsfont}
\usepackage{enumitem}
\usepackage{euscript}
\usepackage{faktor}
\usepackage{hyperref}
\usepackage{float}
\usepackage{geometry}
\usepackage{graphicx}
\usepackage{mathrsfs}
\usepackage{mathtools}
\usepackage{polynom}
\usepackage{stmaryrd}
\usepackage{tikz}
\usepackage{tikz-cd}
\usepackage{wasysym}
\usepackage{xfrac}
\usepackage[x11names]{xcolor}
\usepackage{mathtools}
\usepackage{setspace}
\usepackage{verbatim}

\newif\ifproofread

\DeclarePairedDelimiter\abs{\lvert}{\rvert}

\makeatletter
\newcommand*\bigcdot{\mathpalette\bigcdot@{.5}}
\newcommand*\bigcdot@[2]{\mathbin{\vcenter{\hbox{\scalebox{#2}{$\m@th#1\bullet$}}}}}
\makeatother

\usepackage{mathtools}

\theoremstyle{plain}

\newtheorem*{Lemma*}{Lemma}
\newtheorem*{Corollary*}{Corollary}
\newtheorem*{theorem*}{Theorem}

\newtheorem{claim}{Claim}[section]
\newtheorem{theorem}{Theorem}[section]
\newtheorem{lemma}[theorem]{Lemma}
\newtheorem{Lemma}[theorem]{Lemma}
\newtheorem{proposition}[theorem]{Proposition}
\newtheorem{Corollary}[theorem]{Corollary}

\theoremstyle{definition}

\newtheorem*{definition*}{Definition}

\newtheorem{remark}[theorem]{Remark}

\theoremstyle{plain}
\newcommand{\thistheoremname}{}
\newtheorem*{genericthm}{\thistheoremname}

\newcommand{\NN}{\mathbb{N}}

\newcommand{\QQ}{\mathbb{Q}}

\newcommand{\RR}{\mathbb{R}}

\newcommand{\ZZ}{\mathbb{Z}}

\newcommand{\Dd}{\mathcal{D}}

\newcommand{\fF}{\mathcal{F}}

\newcommand{\gG}{\mathfrak{g}}

\newcommand{\Oo}{\mathcal{O}}

\newcommand{\pP}{\mathfrak{p}}

\newcommand{\Tt}{\mathcal{T}}

\newcommand{\uU}{\mathfrak{u}}

\newcommand{\SL}{\operatorname{SL}}

\newcommand{\diag}{\on{diag}}

\newcommand{\on}[1]{\operatorname{#1}}

\newcommand{\FldGen}{a_K}
\newcommand{\SemiGroup}{\theta^{\ZZ}\zeta^{\NN}}
\newcommand{\zloverp}{\ZZ[\tfrac{1}{p}]}

\newcommand{\pdiag}{A}
\newcommand{\TheLattice}{x_K}

\newcommand{\OurFunDo}{\Dd_0}
\newcommand{\barOurFunDo}{\overline{\Dd_0}}
\newcommand{\Tpbar}{\overline{\Tt_p}}
\newcommand{\pibar}{\overline{\pi}}
\newcommand{\piti}{\tilde{\pi}}
\newcommand{\inner}[2]{\left\langle#1,#2\right\rangle}

\newcommand{\inn}[1]{\left<#1\right>}

\newcommand{\restr}[2]{\ensuremath{\left.#1\right|_{#2}}}

\newcommand{\set}[1]{\left\{#1\right\}}

\newcommand{\pid}{\mathrel{\ooalign{$\lneq$\cr\raise.22ex\hbox{$\lhd$}\cr}}}
\newcommand{\colf}[2]{\begin{pmatrix}
           #1 \\
           #2 
         \end{pmatrix}}

\AtBeginDocument{}
\newcommand{\ve}{\varepsilon}

\DeclareMathOperator{\GL}{GL}

\DeclareMathOperator{\Stab}{Stab}

\newcommand{\manote}[1]{\marginpar{\color{blue}\tiny [MHAG] #1}}

\newcommand{\pvnote}[1]{\marginpar{\color{magenta}\tiny [PV] #1}}

\newcommand{\cmatr}[2]{\left( \begin{matrix} #1 \\ #2 \end{matrix} \right) }

\newcommand{\matr}[4]{\left( \begin{matrix} #1 & #2 \\ #3 & #4 \end{matrix}
\right) }

\newcommand{\cvmatr}[4]{\left(\cmatr{#1}{#2},\cmatr{#3}{#4}\right)}

\newcommand{\av}[1]{\left|#1\right|}
\newcommand{\pa}[1]{\left(#1\right)}

\newcommand{\SubSemigroup}{\theta^{\ell\ZZ} \zeta^{\ell\NN}}
\newcommand{\Closure}{\overline{A(x;v)}}
\newcommand{\ClassNumberNorm}{n_0}
\newcommand{\e}{\underline{e}}

\newcommand{\GLT}{G}
\newcommand{\GLZP}{\Gamma}
\newcommand{\GLZZP }{\Gamma}

\usepackage{enumitem}

\newlist{legal}{enumerate}{10}
\setlist[legal]{label*=\arabic*.}

\title{The Fully Inhomogeneous $p$-Adic Littlewood Conjecture}
\author{
    Menny Aka, Alexander Gorodnik, Pankaj Vishe and Yuval Yifrach
}

\begin{document}
\maketitle
\begin{abstract}We prove that for almost every $\alpha\in\RR$, the \textit{fully inhomogeneous $p$-adic Littlewood Conjecture} holds; namely,
\begin{equation*}
		\forall \delta\in\RR,\forall \kappa\in\ZZ_p, \;\;\liminf_{\av{q}\rightarrow+\infty}q\inn{q\alpha+\delta}\abs{q+\kappa}_p=0.
\end{equation*}
Here, $\inn{\cdot}$ denotes the distance to the nearest integer and $|\cdot|_p$ denotes the $p$-adic norm.
Moreover, we prove this conjecture for every quadratic irrational $\alpha$.
This gives an affirmative answer to the $p$-adic version of a question posed by Cassels in \cite{Cas59}.
\end{abstract}

\section{Introduction}

The interplay between multiplicative Diophantine approximation and homogeneous dynamics has yielded profound insights in both fields. At the core of this interaction lies the classical Littlewood Conjecture, which asserts that for any real numbers $\alpha, \beta \in \R$, 
\begin{equation}\label{eq:littlewood}
\liminf_{q \to +\infty} q \cdot \left<q\alpha\right> \cdot \left<q\beta\right> = 0,
\end{equation}
where $\left<\,\cdot\,\right>$ denotes the distance to the nearest integer. 
While \eqref{eq:littlewood} obviously holds
when at least one of $\alpha,\beta$ is not badly approximable, very little was known otherwise
until the work of Cassels and Swinnerton-Dyer
\cite{Cassels1955},
who established the conjecture 
when $1,\alpha,\beta$ form a $\QQ$-basis for a real cubic field. 
Next important breakthrough was due to Pollington and Velani
\cite{PollingtonVelani2000}
who developed a method for studying 
Diophantine properties  
with respect to
a measure supported on the set of 
badly approximable numbers
and established that \eqref{eq:littlewood}
is also generic in a {strong} sense for
the badly approximable case.
In parallel, Margulis \cite{Margulis1991,Margulis2000}
observed that the Littlewood Conjecture 
is intimately interconnected with the
properties of orbits of multi-parameter actions
on the space of lattices.
This profound observation ultimately 
culminated in the celebrated work of 
Einsiedler, Katok, and Lindenstrauss 
\cite{ekl}, which showed that the set 
of exceptions to the Littlewood Conjecture
has zero Hausdorff dimension.
Interestingly, certain dynamical ideas were
already present in the initial work of
Cassels and Swinnerton-Dyer \cite{Cassels1955},
where the action by the group of units of the cubic number field was exploited.
An analogous action by a group of $S$-units will also play an important 
role in the present paper.

A natural and significant extension of the Littlewood Conjecture was posed by Cassels~\cite{Cas59}. Cassels asked whether the Littlewood Conjecture remains valid under arbitrary inhomogeneous shifts. This fully inhomogeneous formulation, widely known as Cassels' problem, asks if there exist $\alpha, \beta \in \R$ such that for \emph{all} shifts $\delta,\kappa \in \R$,
\begin{equation}\label{eq:littlewood2}
\liminf_{|q| \to \infty} |q| \cdot \left<q\alpha+\delta\right> \cdot \left<q\beta +\kappa\right> = 0.
\end{equation}
For over half a century, it was unknown whether a single pair of numbers satisfied this condition. A complete resolution to Cassel's question was finally achieved by Shapira in \cite{shapiraAnnals}. Using 
dynamics of the space of affine lattices,
Shapira established
that 
for Lebesgue-almost every pair $(\alpha, \beta) \in \R^2$, property \eqref{eq:littlewood2} holds for \emph{all} $\delta,\kappa \in \R$. An explicit rate of convergence for Shapira's result was obtained in a work of Gorodnik and Vishe \cite{AlexPankaj}.
A recent important advance is the work of 
Chow and Zafeiropoulos \cite{CZ2021},
who have developed a far-reaching generalisation of the 
Pollington--Velani framework that allows
to treat inhomogeneous target.
They have established that {a strong version} of \eqref{eq:littlewood2} holds 
when $(\alpha, \beta)$ belongs to large subsets of
badly approximable pairs for \emph{all} $\delta,\kappa \in \R$.

Motivated by these classical real-variable results, a rich parallel theory has been developed, incorporating non-Archimedean metrics. De-Mathan and Teuli\'e \cite{detuli} proposed a $p$-adic analogue of the Littlewood Conjecture, now known as the Mixed Littlewood Conjecture which states 
that 
for any prime number $p$ and  $\alpha \in \R$,
\begin{equation}\label{eq:littlewoodp}
\liminf_{q \to +\infty} q \cdot \left<q\alpha\right> \cdot |q|_p = 0,
\end{equation}
where $| \cdot |_p$ represents the usual $p$-adic absolute value.
De Mathan and Teuli\'e proved their conjecture for all quadratic irrationals. From a metric standpoint, this conjecture is overwhelmingly generic: it obviously holds for Lebesgue-almost all real numbers $\alpha$. 
Adopting the techniques of \cite{ekl},
Einsiedler and Kleinbock
\cite{EinsiedlerKleinbock}
showed that the set of possible
exceptions for $\alpha$
 has zero Hausdorff dimension. 
This approach also led to the result of 
Badziahin, Bugeaud, Einsiedler, and Kleinbock 
\cite{BBEK2015}, who established \eqref{eq:littlewoodp}
when the continued fraction expansion of $\alpha$ has exponential complexity. 
Furthermore, Haynes, Jensen, and Kristensen \cite{HJK2014} generalized the framework of 
\cite{PollingtonVelani2000} to this setting.
Despite these remarkable advances, the search for an explicit counterexample also remains an active frontier. 

As with the classical setting, the Mixed Littlewood Conjecture has naturally evolved to encompass inhomogeneous targets. Bugeaud, Drmota, and de Mathan \cite{BDM2007} conjectured that 
for every $\alpha \in \R$ and $\kappa \in \Zp$,
\begin{equation}\label{eq:littlewoodpp}
\liminf_{q \to +\infty} q \cdot \left<q\alpha\right> \cdot |q +\kappa|_p = 0.
\end{equation}
They  demonstrated validity of \eqref{eq:littlewoodpp} for specific uncountable classes of badly approximable numbers characterized by combinatorial word patterns in their continued fraction expansions. 
Furthermore, it is 
natural to consider doubly/fully inhomogeneous targets analogue of \eqref{eq:littlewoodp}, which is the main focus of this work.
The main result of the paper is a \textit{mixed/p-adic} analogue
of the main result of \cite{shapiraAnnals}:

\begin{theorem}\label{thm: main}
    For almost every $\alpha\in \R$, 
    \begin{equation}\label{eq: inhomogenuous liminf intro}
\forall \delta\in[0,1) \textrm{ and }\;\forall \kappa\in\ZZ_p,\;\;\liminf_{\av{q}\rightarrow +\infty}q\inn{q\alpha+\delta}\abs{q+\kappa}_p=0.
	\end{equation}
    Moreover \eqref{eq: inhomogenuous liminf intro} holds for every quadratic irrational $\alpha$.
\end{theorem}


\noindent\textbf{Acknowledgements:} M.A. was supported by SNF Grant number 10003145. 
A.G. and Y.Y. were supported by SNF grant 200020--212617.
In the early part of the project, P.V. was supported by EPSRC grant
EP/J018260/1. The authors would like to thank Manfred Einsiedler for useful discussions.

\subsection{Organisation of the paper}
Section \ref{subsec: nota} introduces the basic notation used in the paper.
In particular, the central role will be played by the space $X$ consisting of free rank-two covolume-one $\Z[\frac{1}{p}]$-modules in $\R^2\times\Q_p^2$
and the space $Y$ consisting of affine version of such modules,
which can be  naturally viewed as a bundle over $X$. 
In Section \ref{sec: dani},
we develop a correspondence which 
allows us to reformulate the original Diophantine property
in terms of orbit behaviour on the space $Y$. 
In Section \ref{sec: periodic}, we construct
particular elements $x_K$ in $X$ 
arising from quadratic number fields $K$
and observe that those elements are stabilised
by actions of certain higher-rank abelian semigroups.
The fibers of the elements $x_K$ in $Y$ are $p$-adic tori. 
In Section \ref{sec: dense orbits}, we study
actions of the higher-rank abelian semigroups
on the $p$-adic tori.
Finally, we prove the first part of the main theorem
in Section \ref{sec:gen} and 
the second part in Section \ref{sec: quad}.

\section{Notation}\label{subsec: nota}
Throughout the paper, $p$ denotes an arbitrary fixed prime number.
Let $X_\infty$ denote the space of
of two-dimensional lattices of covolume one in $\RR^2$, which
can be parametrized as 
$X_\infty\cong\on{SL}_2(\RR)/\on{SL}_2(\ZZ)$.
Similar parametrization also exists for 
free rank-two $\zloverp$-submodules in $\RR^2\times \QQ_p^2$, where a $\zloverp$-module structure on $\RR^2\times \QQ_p^2$ is defined by 
the diagonal embedding of $\zloverp$ into $\RR\times\QQ_p$. 
We call such modules   
 $\zloverp$-lattices (or simply lattices) in $\RR^2\times \QQ_p^2$. They are of the form
$$\Lambda=
\zloverp(v_1^\infty,v_1^{(p)})\oplus \zloverp (v_2^\infty,v_2^{(p)})\subset \RR^2\times \QQ_p^2,
$$ 
with
 $v_1^\infty,v_2^\infty\in \RR^2$ and $v_1^{(p)},v_2^{(p)}\in \QQ_p^2$, where
 the matrices $g_\infty$ and $g_p$ formed
 by the columns $v_1^\infty,v_2^\infty$ and
 $v_1^{(p)},v_2^{(p)}$ respectively satisfy
$$\det(g_\infty,g_p):=\abs{\det(g_\infty)}\abs{\det(g_p)}_p\neq 0. $$
 Here, the absolute values are the real and the $p$-adic ones respectively.
We say that $\Lambda$ has co-volume one if $\det(g_\infty,g_p)=1$. 
Let 
$$
\Gamma:=\GL_2(\zloverp)
$$ be the group consisting of matrices with entries in $\zloverp$ with determinant in $\zloverp^\times=\set{\pm p^n\colon n\in \ZZ}$.
We view $\Gamma$ as a subgroup of $\GL_2(\RR)\times\GL_2(\QQ_p)$
embedded diagonally. Then $\det(\gamma,\gamma)=1$ 
for any $\gamma\in \Gamma$.
We note that the lattice $\Lambda$
determines the pair $(g_\infty,g_p)$
uniquely up to multiplying 
$(g_\infty,g_p)$ on the right by an element of $\Gamma$. 
Therefore, the following homogeneous space serves as a moduli space for such covolume-one lattices  
\begin{equation}\label{eq:TheSpace}
    X:=G/\Gamma,
\end{equation}
where 
$$G:=
\set{(g_\infty,g_p)\in \GL_2(\RR)\times \GL_2(\QQ_p):\det(g_\infty,g_p)=1}.$$

Let $I$ denote the $2\times 2$ identity matrix and $\diag(a,b)$ denote the $2\times 2$ diagonal matrix with diagonal entries $a,b$.   
We claim that the space $X$ decomposes as the union of disjoint closed $\on{SL}_2(\RR\times\QQ_p)$-orbits
\begin{equation}
    X_u:=\on{SL}_2(\RR\times\QQ_p)(I,\diag(1,u))\GLZP,
\end{equation}
parametrized by $u\in \ZZ_p^\times$.
Indeed, any $x\in X$ can be written as  $x=(g_\infty,g_p)\GLZP$ with $\det g_\infty=1$ and $\av{\det g_p}_p=1$ by multiplying the coset representative on the right by $(\gamma,\gamma)\in \GLZP $ where $\gamma=\diag ((\det g_\infty)^{-1},1)$. Setting $u=\det g_p$, we see that
$x$ is contained in the $\on{SL}_2(\RR\times\QQ_p)$-orbit $X_u$. The sets $X_u$ are closed because they are fibers
of the map $(g_\infty,g_p)\Gamma\mapsto (\det g_\infty)|\det g_p|_p$. We note that $X_u$ is isomorphic to $\on{SL}_2(\RR\times\QQ_p)/\Gamma_u$ for a lattice $\Gamma_u$ in $\on{SL}_2(\RR\times\QQ_p)$. 

For $x\in X$, we denote by $\Lambda_x$ the corresponding lattice in $\RR^2\times \QQ_p^2$. 
Affine lattices (or grids) in $\RR^2\times \QQ_p^2$ are the sets of the form $\Lambda+v$ with 
a lattice $\Lambda\subset\RR^2\times \QQ_p^2$ and $v\in \RR^2\times \QQ_p^2$.
The set of affine covolume-one lattices are parametrized
by the homogeneous space
\begin{equation}\label{eq:Y}
    Y:=\left(\GLT\ltimes (\RR^2\times\QQ_p^2)\right)/\left(\GLZP \ltimes \zloverp^2\right).
\end{equation}
We note that the map $\Lambda+v\mapsto \Lambda$
defines a bundle structure on $Y$ with the base $X$ and  
the fibers over $x\in X$ being the $p$-adic tori
$(\RR^2\times \QQ_p^2)/\Lambda_x$.


\section{$p$-adic Dani correspondence}\label{sec: dani}


In this section, we develop a connection between the inhomogeneous $p$-adic Littlewood problem and properties
of orbits of the semigroup 
\begin{equation}\label{eq:A} 
		\pdiag :=\left\{\left(p^{-n}\begin{pmatrix}
			e^t & 0\\
			0   & e^{-t}
		\end{pmatrix},
		\begin{pmatrix}
			a & 0\\
			0 & a^{-1}p^{-2n}
		\end{pmatrix}\right):t\in \RR,a\in \ZZ_p^{\times},n\in \NN\right\}
	\end{equation}
acting on the space $Y$. It turns out that solving Diophantine inequalities in this setting can be 
reduced to establishing that the corresponding $A$-orbits visit certain subsets of $Y$.

Throughout this section, we will treat elements of $\RR^2\times \QQ_p^2$ as pairs of column vectors
and introduce a function $N$ on $\RR^2\times \QQ_p^2$:
	\begin{equation}
		N\left(\colf{x_1}{x_2},\colf{x_3}{x_4}\right):=\abs{x_1x_2}\cdot \abs{x_4}_p.
	\end{equation}	
Let 
\begin{align*}
\mathcal{O}_\varepsilon&:=
\left\{x\in  (-\sqrt{\varepsilon},\sqrt{\varepsilon})^2 \times \ZZ_p^2:\, x_1,x_2,x_4\ne 0\right\}, \\
\Omega_\varepsilon&:=\{\Lambda\in Y: \Lambda\cap \mathcal{O}_\varepsilon\ne \emptyset\}.
\end{align*}

\begin{lemma}\label{l:NA} 
Let $x\in \RR^2\times \QQ_p^2$.
\begin{enumerate}
\item[(a)]  If $Ax\cap \mathcal{O}_\varepsilon\ne \emptyset$, then $0<N(x)<\varepsilon$.
\item[(b)]  If $0<N(x)<\varepsilon$ and $x_3\in\ZZ_p$, $0\neq x_4\in\ZZ_p$, then
$Ax\cap \mathcal{O}_{p^2\varepsilon}\ne \emptyset$.
\end{enumerate}    
\end{lemma}

\begin{proof}
Part  $(a)$ is straightforward because the function $N$ is $A$-invariant. For Part $(b)$,
we choose $n_0\in\NN$ and $t_0\in\RR$ such that
$p^{-2}<p^{2n_0}|x_4|_p\le 1$ and $e^{t_0}|x_1|=e^{-t_0}|x_2|$. Let $a\in A$ be any element corresponding to these parameters $n_0$ and $t_0$ as appearing in \eqref{eq:A} .
Then $x':=ax$ satisfies
$
0<|x'_1|,|x'_2|< \sqrt{p^2\varepsilon},$
$|x'_3|_p\le 1,$ and
$0<|x'_4|_p\le 1,$
so that $x'\in \mathcal{O}_{p^2\varepsilon}$.
\end{proof}

For $\overline v=(v_1,v_2,v_3)\in \RR^2\times \QQ_p$ and for $u\in \ZZ_p^\times$, we introduce the grid
	\begin{equation}\label{eq:Lambda}
		\Lambda_{\overline v,u}:=\left(\begin{pmatrix}
			1 & 0\\
			v_1 & 1
		\end{pmatrix},\begin{pmatrix}
			u & 0\\
			0 & 1
		\end{pmatrix}\right)\zloverp^2+\left(\colf{0}{v_2},\colf{0}{v_3}\right).
	\end{equation}
With this notation we formulate the main result, which connects the original arithmetic question to a dynamical problem involving $A$-orbits of the grids  \eqref{eq:Lambda} visiting the subsets $\Omega_\varepsilon$.

\begin{proposition}[$p$-adic Dani correspondence]
\label{p:dani}
Let $\alpha\in (0,1)$, $\delta\in[0,1)$,  $\kappa\in\ZZ_p$. We set 
$\overline{v}:=(1/\alpha,\delta/\alpha,\kappa)$. Then:
\begin{enumerate}
    \item[(a)] If for $\varepsilon\in (0,\varepsilon_0(\alpha,\delta,\kappa))$, 
    $$
    A\cdot \Lambda_{\overline v,u}\cap \Omega_\varepsilon\ne \emptyset
    \quad\hbox{for some $u\in \ZZ_p^\times$,}
    $$
    then
    there exists $q\in \ZZ\backslash \{0\}$ such that 
    $$
   0<|q|\inn{q\alpha+\delta}|q+\kappa|_p<3\varepsilon. 
    $$
    \item[(b)] Conversely, 
    if 
    $$
   0<|q|\inn{q\alpha+\delta}|q+\kappa|_p<\varepsilon 
    $$
    has a solution $q\in \ZZ\backslash \{0\}$ for some 
    $\varepsilon\in (0,\varepsilon_0(\alpha,\delta,\kappa))$, 
    then
    $$
    A\cdot \Lambda_{\overline v,u}\cap \Omega_{p^2(1+2\alpha^{-1})\varepsilon}\ne \emptyset\quad\hbox{for every $u\in \ZZ_p^\times$.
}
    $$
\end{enumerate}
\end{proposition}

\begin{proof}
We begin by verifying the assertion $(a)$. Suppose that for some $a\in A$ and $u\in \ZZ_p^\times$,
we have
$a\cdot \Lambda_{\overline v,u}\in \Omega_\varepsilon$.
Recalling \eqref{eq:A} and \eqref{eq:Lambda}, we see 
that the union of the grids $a\cdot \Lambda_{\overline v,u}$
consists of the elements
\begin{align*}
x=\cvmatr{p^{-n}e^{t}q_0}{p^{-n}e^{-t}(q_0v_1+q_1+v_2)}{bu q_0}{p^{-2n}b^{-1}(q_1+v_3)},\quad\hbox{$q_0,q_1\in \zloverp$},
\end{align*}
with $n\in\NN$, $t\in\RR$, and $b\in \ZZ_p^\times$.
We deduce from Lemma \ref{l:NA}$(a)$  that there exists an $x$ satisfying 
\begin{equation*}
0 <N(x)= |q_0|\cdot |q_0v_1+q_1+v_2|\cdot |q_1+v_3|_p <\varepsilon\quad\hbox{and}\quad x_3,x_4\in \ZZ_p. 
\end{equation*}
A priori, we only know that $q_0,q_1\in \zloverp$. 
Additionally,
since 
$$
x_3=buq_0\in \ZZ_p,
$$
we conclude that 
$q_0\in\zloverp\cap\ZZ_p=\ZZ$. Similarly, 
since 
$$
x_4=p^{-2n}b^{-1}(q_1+v_3)\in\ZZ_p
\quad\hbox{and}\quad v_3=\kappa\in\ZZ_p,
$$
we obtain that $q_1\in\ZZ$.
Therefore, it follows  that 
\begin{equation}\label{eq:1__}
0<|q_0|\cdot |q_0+q_1\alpha+\delta|\cdot |q_1+\kappa|_p <\alpha\varepsilon
\end{equation}
has a solution $q_0\in \ZZ\backslash \{0\}$, $q_1\in \ZZ$.

Suppose that we have such a solution with $q_1=0$.
Then $\kappa\ne 0$.
If $\delta =0$, this product is at least $|\kappa|_p$. 
If $\delta \in (0,1)$, this product is at least $\left<\delta\right>|\kappa|_p$.
We now set
\begin{equation*}
    \ve_0(\alpha,\delta,\kappa):=\begin{cases}
        |\kappa|_p/2,&\textrm{ if }\delta=0,\\
        \langle\delta\rangle|\kappa|_p/2,&\textrm{ otherwise. }
    \end{cases}
\end{equation*}
Hence, if $\varepsilon<\ve_0(\alpha,\delta,\kappa)$, then a solution $q_0,q_1\in\ZZ$ to \eqref{eq:1__} must further satisfy $q_1\ne 0$. We work with such an $\varepsilon_0$ throughout now.

We analyze following two cases separately.
Suppose that $|q_0+q_1\alpha+\delta|\leq \frac12$. 
Then $|q_0+q_1\alpha+\delta|=\langle q_1 \alpha+\delta\rangle$, and the triangle inequality implies that $\alpha|q_1| \le 3|q_0|$. Thus, we conclude that 
\begin{equation}\label{eq:2}
0<|q_1| \langle q_1 \alpha+\delta\rangle |q_1+\kappa|_p <3\varepsilon,
\end{equation}
as required. 

Now suppose that $|q_0+q_1\alpha+\delta|> \frac12$.
In this case, we claim that
\begin{equation}\label{eq:q0q1gamma1}
|q_0||q_0+q_1\alpha+\delta|\ge \alpha|q_1|/6.
\end{equation}
If $|q_0|\geq \alpha|q_1|/3$, then our assumption clearly implies \eqref{eq:q0q1gamma1}.
Furthermore, if $|q_1|\leq  3/\alpha$, then  \eqref{eq:q0q1gamma1} again follows from our assumption.
In the remaining case, when $|q_1|\geq 3/\alpha$ and $|q_0|\leq \alpha|q_1|/3$, the triangle inequality implies that
\begin{equation*}
    |q_0+q_1\alpha+\delta|\geq \alpha|q_1|-|q_0|-1\geq \alpha|q_1|/3,
\end{equation*}
which also implies \eqref{eq:q0q1gamma1}.
Ultimately, we deduce from \eqref{eq:1__} and \eqref{eq:q0q1gamma1} that 
$$
0<|q_1|\cdot |q_1+\kappa|_p <6\varepsilon,
$$
and since $\langle q_1\alpha+\delta\rangle\in [0,1/2]$, this implies \eqref{eq:2} and verifies part (a).

Next we deal with part (b). Our assumption implies that 
there exist $q_0\in\ZZ$ and $q_1\in\ZZ\backslash \{0\}$ such that 
\begin{equation}\label{eq:3}
   0<|q_1|\cdot |q_0+q_1\alpha+\delta|\cdot |q_1+\kappa|_p<\varepsilon\quad\hbox{and}\quad
   |q_0+q_1\alpha+\delta|\le 1/2.
\end{equation}
The triangle inequality implies that $|q_0|\le (\alpha+2)|q_1|$, so that we derive the bound
$$
   |q_0|\cdot |q_0v_1+q_1+v_2|\cdot |q_1+v_3|_p<(1+2\alpha^{-1})\varepsilon.
$$
Suppose that the only solutions \eqref{eq:3} arise when $q_0=0$. In this case, we observe that $q_1$ must satisfy $|q_1|\leq 2/\alpha$, and this gives only finitely many choices for $q_1$. Therefore, the set $\{0\neq |q_1|\cdot |q_1\alpha+\delta|\cdot |q_1+\kappa|_p:|q_1|\leq 2/\alpha\}$ is finite and has a finite lower bound, which we call $\ve_0(\alpha,\delta,\kappa)$. Note that for any
$\varepsilon<\ve_0(\alpha,\delta,\kappa)$ as above, a solution to \eqref{eq:3} must have $q_0\ne 0$.
This implies that the grid $\Lambda_{\bar v,u}$ 
contains an element $y$ such that 
$$
0<N(y)=|y_1y_2|\cdot |y_4|_p< (1+2\alpha^{-1})\varepsilon,
\quad
|y_3|_p\le 1, \quad 0\neq |y_4|_p\leq 1.
$$
By Lemma \ref{l:NA}$(b)$,  
$A\cdot \Lambda_{\bar v,u}\cap  \Omega_{p^2(1+2\alpha^{-1})\varepsilon}\ne \emptyset$.
This verifies $(b)$.
\end{proof}
As a corollary, we obtain 
\begin{Corollary}\label{lem: dani}
Given $\alpha\in (0,1)$, $\delta\in[0,1)$, $u_0\in \ZZ_p^\times$ and $\kappa\in\ZZ_p$, let $\overline{v}:=(1/\alpha,\delta/\alpha,\kappa)\in\RR^2\times\ZZ_p$. If
    \begin{equation*}
		\inf\set{0\neq N(u)\colon u=\cvmatr{u_1}{u_2}{u_3}{u_4}\in A \Lambda_{\overline v,u_0},  u_3,u_4\in\ZZ_p}=0,
	\end{equation*}
    then
	\begin{equation*}\label{eq:Victory}
\liminf_{|q|\rightarrow\infty}|q|\inn{q\alpha+\delta}|q+\kappa|_p=0.
	\end{equation*}
\end{Corollary}

\section{Compact diagonal orbits in $X$}
\label{sec: periodic}

Let $X$ denote the space of two-dimensional $\zloverp$-lattices of co-volume one as defined in \eqref{eq:TheSpace}. Compact diagonal orbits in $X$ will play a key role in our argument. In the case of two-dimensional real lattices, the compact diagonal orbits correspond to orders in real quadratic fields. We will give an explicit analogous construction for the space $X$.

Let $K$ be a real quadratic number field and
$\Oo_K$
its ring of integers.
We assume that the prime $p$ splits in $K$, namely, $p\Oo_K = \pP_1\pP_2$ for  ideals $\pP_i$, $i=1,2$.
Set $h$ to be the class number of $K$, so that $\pP_i^h = \tau_i \Oo_K$, $i=1,2$, are principal ideals. 
We choose $\tau_i$, so that they are Galois conjugates of each other. 
Let $\sigma_i^\infty: K \to \RR$, $i=1,2$ (resp. $\sigma_i^p: K \to \QQ_p$, $i=1,2$),  be two distinct embeddings into $\RR$ (resp.\ $\QQ_p$). Note that for $i=1,2$, $\sigma_i^p$ is the embedding of $K$ associated to the completion with respect to the ideal $\pP_i$.

We denote by $\Oo_{K}^\times$ the group of units of 
$\Oo_{K}$ and fix the fundamental unit 
$\epsilon_\infty$  satisfying $\pm\langle \epsilon_\infty \rangle = \Oo_K^\times$ and $\sigma_1^\infty(\epsilon_\infty) > 1$.
We set $S = \{\pP_1, \pP_2\}$ and consider the ring $\Oo_{K,S}$ of $S$-integers (see \cite[Page 90]{MilneANT}) defined by
\begin{equation}
    \Oo_{K,S}:=\{a\in K: \mathrm{ord}_{\pP}(a)\geq 0 \textrm{ for all prime ideals } \pP\notin S\}.
\end{equation}
The group of $S$-units $\Oo_{K,S}^\times$ is further defined as
\begin{equation}
    \Oo_{K,S}^\times:=\{a\in K: \mathrm{ord}_{\pP}(a)= 0 \textrm{ for all prime ideals } \pP\notin S\}.
\end{equation}
 As proved in loc.~cit.~or in \cite[Cor.\ II.7.3]{NeukirchANT}, $\Oo_{K,S}^\times$ has rank $3$ and $\langle \epsilon_\infty, \tau_1, \tau_2 \rangle \subseteq \Oo_{K,S}^\times$ is a rank-$3$ subgroup.  For $t \in \Oo_{K,S}^\times$, we set 
$$\sigma(t) := \Big(\diag\big(\sigma_1^\infty(t), \sigma_2^\infty(t)\big), \diag\big(\sigma_1^p(t), \sigma_2^p(t)\big)\Big).$$

Since $K$ is quadratic, one can choose $\FldGen$ such that $\ZZ[\FldGen] = \Oo_K$. 
We now set
\begin{equation}
    \label{eq:THELattice}
\TheLattice:=  \left( d_K^{-\frac12}\begin{pmatrix} 1 & \sigma_1^\infty(\FldGen) \\ 1 & \sigma_2^\infty(\FldGen) \end{pmatrix},  \begin{pmatrix} 1 & \sigma_1^p(\FldGen) \\ 1 & \sigma_2^p(\FldGen) \end{pmatrix} \right) \GLZZP,
\end{equation}
where $d_K:=\abs{\sigma_2^\infty(\FldGen)-\sigma_1^\infty(\FldGen)}$.
Note that since $p$ splits, it does not divide the discriminant of $\ZZ[\FldGen]$ which by definition is equal to $\left(\sigma_1^p(\FldGen)-\sigma_2^p(\FldGen)\right)^2$. Therefore, 
$\sigma_2^p(\FldGen)-\sigma_1^p(\FldGen)\in \ZZ_p^\times$,
and $\TheLattice$ belongs to the space $X$.
Moreover,
\begin{equation}\label{eq: u def}
   \TheLattice\in X_{u_K}, \textrm{ where } u_K:=\sigma_2^p(\FldGen)-\sigma_1^p(\FldGen)\in\ZZ_p^\times.
\end{equation}

\begin{Lemma}\label{lem: NEW two stabilizers}
    There exist $t_0>0,s_0\in \RR$, $n_0\in \ZZ_{>0}$, and $ a,b\in \ZZ_p^{\times}$ with $a$ having infinite order such that the elements
	\begin{equation}\label{eq:alpha def}
		\theta:= \sigma(\epsilon_\infty^2)=\left(\begin{pmatrix}
			e^{t_0} & 0\\
			0   & e^{-t_0}
		\end{pmatrix},
		\begin{pmatrix}
			a^{-1} & 0\\
			0 & a
		\end{pmatrix}\right),
	\end{equation} 
	\begin{equation}\label{eq:beta def}
		\zeta := \sigma(\tau_2^{-2})=\left(p^{-\ClassNumberNorm}\begin{pmatrix}
			e^{s_0} & 0\\
			0   & e^{-s_0}
		\end{pmatrix},
		\begin{pmatrix}
			b & 0\\
			0 & b^{-1}p^{-2\ClassNumberNorm}
		\end{pmatrix}\right)
	\end{equation}
	belong to $\mathrm{Stab}_A(\TheLattice)$, and the pairs $\{e^{t_0}$, $p^{-\ClassNumberNorm}e^{s_0}\}$, $\{e^{t_0}$, $p^{-\ClassNumberNorm}e^{-s_0}\}$, $\{e^{2s_0},e^{2t_0}\}$ are multiplicatively independent.  
\end{Lemma}
\begin{proof}
We begin by recalling that (see \cite[Ch.\ II, (8.4)]{NeukirchANT}) for any $y \in K$,
\begin{equation}\label{eq:NormInComp}
\sigma_1^\infty(y)\sigma_2^\infty(y) = \operatorname{Nr}_{K/\QQ}(y) = \sigma_1^p(y)\sigma_2^p(y).    
\end{equation}
 Since $\operatorname{Nr}_{K/\QQ}(\epsilon_\infty^2) = 1$, we conclude that $\theta=\sigma(\epsilon_\infty^2)\in\SL_2(\RR)\times \SL_2(\QQ_p)$.  We define $t_0$ via $e^{t_0} = \sigma_1^\infty(\epsilon_\infty^2)$. Then $t_0 > 0$ since $\sigma_1^\infty(\epsilon_\infty) > 1$. 
 We set $a := \sigma_2^p(\epsilon_\infty^2)$.
 Since $\epsilon_\infty^2$ has infinite order, $a$
 also has infinite order.
It follows from \eqref{eq:NormInComp} that
$\sigma_1^p(\epsilon_\infty^2)=a^{-1}$.
 Since $\epsilon_\infty\in \Oo_K^\times$, we have $\epsilon_\infty\notin \pP_i$, so that
 $a\in \ZZ_p^\times$. 
This verifies that $\theta$ is of the form \eqref{eq:alpha def}.

To prove that $\zeta$ has the form in  \eqref{eq:beta def}, we recall that $\tau_i \Oo_K  = \pP_i^h$ where $p\Oo_K = \pP_1\pP_2$. As $\tau_1$ and $\tau_2$ are Galois conjugates of each other, they have the same norm. Therefore, for $i = 1,2$, we have $\operatorname{Nr}_{K/\QQ}(\tau_i) = \pm p^h$ and $\operatorname{Nr}_{K/\QQ}(\tau_i^{2}) = p^{2h}$.
Since $\sigma_i^p$ 
sends $K$ to the completion with respect to $\pP_i$, $\sigma_1^p(\tau_1^{2}) = p^{2h} b_1$ for some $b_1 \in \ZZ_p^\times$. Then it follows from \eqref{eq:NormInComp} that $\sigma_2^p(\tau_1^{2}) = b_1^{-1}$.
Similarly, $\sigma_2^p(\tau_2^{2}) = p^{2h} b_2$ for $b_2 \in \ZZ_p^\times$, and $\sigma_1^p(\tau_2^{2}) = b_2^{-1}$.
We now set $b = b_2$, and $\ClassNumberNorm = h$.  Since
\begin{equation}\label{eq:PNormBeta}
\sigma_1^\infty(\tau_2^{-2}) \, \sigma_2^\infty(\tau_2^{-2}) = \operatorname{Nr}_{K/\QQ}(\tau_2^{-2}) = p^{-2h},    
\end{equation}
we define $s_0$ to be the real number satisfying
\begin{equation}
\label{eq:s0def}  
p^{\ClassNumberNorm} 
\diag\big(\sigma_1^\infty(\tau_2^{-2}),\sigma_2^\infty(\tau_2^{-2})\big)=\diag\big(
e^{s_0},e^{-s_0}\big).
\end{equation}
This is possible since $\sigma_1^\infty(\tau_2^{-2})=\sigma_1^\infty(\tau_2^{-1})^2>0$. This shows that $\zeta$ has the form  \eqref{eq:beta def}.

We now turn to the multiplicative independence. Taking images under a field embedding preserves multiplicative independence. The pairs $\{e^{t_0}$, $p^{-\ClassNumberNorm}e^{s_0}\}$, $\{e^{t_0}$, $p^{-\ClassNumberNorm}e^{-s_0}\}$, and $\{e^{2s_0},e^{2t_0}\}$
are respectively the images of $\set{\epsilon_\infty^2,\tau_2^{-2}}$ under $\sigma_1^\infty$, $\set{\epsilon_\infty^{-2},\tau_2^{-2}}$ under $\sigma_2^\infty$, and $\set{\tau_1^2\tau_2^{-2},\epsilon_\infty^4}$ under $\sigma_1^\infty$. 
To check the last case, we use that 
since $\tau_1$ and $\tau_2$ are Galois conjugate,
$\sigma_1^\infty(\tau_2)=\sigma_2^\infty(\tau_1)$.
Hence, the multiplicative independence follows
from the independence of $\epsilon_\infty,\tau_1,\tau_2$.

Finally, we show that any element of the form $\sigma(t)$, $t \in \Oo_{K,S}^\times$, stabilises $\TheLattice$. As in the real case, this follows from the fact that multiplication by $\Oo_{K,S}^\times$ preserves $\Oo_{K,S}$. 
We carry out the computation below for completeness.
For $t \in \Oo_{K,S}^\times$,
\begin{align}\label{eq:1}
\sigma(t)\TheLattice 
&=
\left(d_K^{-\frac12}
\left(\begin{smallmatrix}
\sigma_1^\infty(t) & \sigma_1^\infty(t\FldGen) \\
\sigma_2^\infty(t) & \sigma_2^\infty(t\FldGen)
\end{smallmatrix}\right),
\left(\begin{smallmatrix}
\sigma_1^p(t) & \sigma_1^p(t\FldGen) \\
\sigma_2^p(t) & \sigma_2^p(t\FldGen)
\end{smallmatrix}\right)
\right)
\GLZZP .
\end{align}
Since $t \in \Oo_{K,S}^\times$ and $\{1, \FldGen\}$ is a basis of $\Oo_{K,S}=\Oo_{K}\otimes\zloverp$ as a $\ZZ[\tfrac{1}{p}]$-module, so is $\{t, t \FldGen\}$. Therefore, there exists a change of $\zloverp$-basis matrix 
\begin{equation}
    \label{eq:ChangeOfBasisMat}
\gamma =\gamma(t):=
\begin{pmatrix}
a & b \\
c & d
\end{pmatrix}
\in \mathrm{M}_2(\ZZ[\tfrac{1}{p}])
\quad \text{with} \quad
t = a + c \FldGen, \quad
t \FldGen = b + d \FldGen,
\end{equation}
such that $\det(\gamma)=\pm p^\ell$ for some $\ell\in \ZZ$. Then $\det(\gamma,\gamma)=p^{\ell}p^{-\ell} = 1$,
so that $(\gamma,\gamma)\in\Gamma$. 
It follows that the right-hand side expression in \eqref{eq:1} is equal to
\[
\sigma(t)\TheLattice =\left(d_K^{-\frac{1}{2}}
\begin{pmatrix}
1 & \sigma_1^\infty(\FldGen) \\
1 & \sigma_2^\infty(\FldGen)
\end{pmatrix},
\begin{pmatrix}
1 & \sigma_1^p(\FldGen) \\
1 & \sigma_2^p(\FldGen)
\end{pmatrix}
\right)
(\gamma, \gamma) \GLZZP  = \TheLattice, 
\]
as required.
\end{proof}
\begin{remark}\label{rem:EVofChangeBasisMat}
    For later use, we note that $\gamma=\gamma(t)$ defined above in \eqref{eq:ChangeOfBasisMat} and $\diag\big(\sigma_1^\infty(t),\sigma_2^\infty(t)\big)$ are conjugates over $\RR$.
\end{remark}

Throughout the paper, we fix the $\zloverp$-lattice corresponding to $\TheLattice$ defined in \eqref{eq:THELattice}  by $\Lambda_p:=\Lambda_{\TheLattice}$, set
\begin{equation}\label{eq:Ttpdef}
    \Tt_p := (\RR^2\times\QQ_p^2)/\Lambda_p.
\end{equation}
and denote by  $\pi: \RR^2\times\QQ_p^2\rightarrow \Tt_p$,
the factor map.

Let 
\begin{align}
 \e_1&=\left((1,0)^t,(0,0)^t\right),\quad
 \e_2=\left((0,1)^t,(0,0)^t\right),\label{eq:vector}\\
 \e_3&=\left((0,0)^t,(1,0)^t\right),\quad
 \e_4=\left((0,0)^t,(0,1)^t\right).\nonumber   
\end{align}
denote the standard basis of $\RR^2\times\QQ_p^2$.
The lattice $\Lambda_p$ is explicitly given as the free two-dimensional $\mathbb{Z}[\tfrac1p]$-module generated by the vectors
\begin{align} \label{eq:x12def}
\mathbf{x}_1&:=d_K^{-\frac{1}{2}}\e_1+d_K^{-\frac{1}{2}}\e_2+
\e_3+\e_4,\\
\mathbf{x}_2&:=d_K^{-\frac{1}{2}}\sigma_1^\infty(\FldGen)\e_1+d_K^{-\frac{1}{2}}\sigma_2^\infty(\FldGen)\e_2+
\sigma_1^p(\FldGen)\e_3+\sigma_2^p(\FldGen)\e_4.\nonumber
\end{align}
The space $\Tt_p$ is isomorphic to the standard solenoid $\Omega_p^2:=(\RR^2\times\QQ_p^2)/\zloverp^2$, where the required isomorphism corresponds to the change of basis from the standard basis to the basis $\mathbf{x}_1$ and $\mathbf{x}_2$. The standard fundamental domain $[0,1)^2\times\ZZ_p^2$ for $\Omega_p^2$ transforms into a fundamental domain
\begin{equation}\label{eq:d0def}
    \OurFunDo:=\left( d_K^{-\frac12}\begin{pmatrix} 1 & \sigma_1^\infty(\FldGen) \\ 1 & \sigma_2^\infty(\FldGen) \end{pmatrix},  \begin{pmatrix} 1 & \sigma_1^p(\FldGen) \\ 1 & \sigma_2^p(\FldGen) \end{pmatrix} \right)\left([0,1)^2\times \ZZ_p^2\right).
\end{equation}
 As noted in \eqref{eq: u def}, the determinant of the matrix $ \begin{pmatrix} 1 & \sigma_1^p(\FldGen) \\ 1 & \sigma_2^p(\FldGen) \end{pmatrix}$ is in $\ZZ_p^\times$, making it invertible over $\ZZ_p$, so that
 \begin{equation}
    \label{eq:d0property}
\OurFunDo= \pa{d_K^{-\frac12}\begin{pmatrix} 1 & \sigma_1^\infty(\FldGen) \\ 1 & \sigma_2^\infty(\FldGen) \end{pmatrix}[0,1)^2}\times \ZZ_p^2\subset \RR^2\times\ZZ_p^2.\end{equation}

We also introduce the space
\begin{equation}\label{eq:Tpdef}
    \Tpbar:=\Tt_p/V_p,\quad\hbox{where}\;\;
    V_p:=\pi(\ZZ_p \e_3)\subset \Tt_p.    
\end{equation} 
Then 
\begin{equation}    \label{eq:d0property1}
\barOurFunDo:= \pa{d_K^{-\frac12}\begin{pmatrix} 1 & \sigma_1^\infty(\FldGen) \\ 1 & \sigma_2^\infty(\FldGen) \end{pmatrix}[0,1)^2}\times \begin{pmatrix} 0\\ \ZZ_p\end{pmatrix}\subset \RR^2\times\ZZ_p^2\end{equation}
is a fundamental domain for $\Tpbar$.
We denote by 
$$
\pibar:\RR^2\times \QQ_p^2\rightarrow \Tpbar\quad\hbox{and}\quad \piti:\Tt_p\rightarrow \Tpbar 
$$
the corresponding quotient maps. 

Before moving on, we check a property that the closures of certain one-dimensional lines in $\Tt_p$ are dense. 

\begin{lemma}\label{lem:Irrationality}
 Each of the sets
\begin{equation}\label{eq:realLines}
V_i:=\{\pi(s\e_i):s\in \RR_{>0}\}, \quad i = 1, 2,    \end{equation}
\begin{equation}\label{eq:padicLines}
    V_i:=\{\pi(s\e_i):s\in \QQ_p\}, \quad i = 3, 4,
\end{equation}
 is dense in $\Tt_p$.
\end{lemma}

\begin{proof}
It follows from \eqref{eq:x12def} that
\[
 \Lambda_p=\set{(M_\infty q,M_pq):q\in\zloverp^2},
\]
where 
\[
M_\infty:=d_K^{-1/2}
\begin{pmatrix}
1&\sigma_1^\infty(\FldGen)\\
1&\sigma_2^\infty(\FldGen)
\end{pmatrix}
\quad\hbox{and}\quad
M_p:=
\begin{pmatrix}
1&\sigma_1^p(\FldGen)\\
1&\sigma_2^p(\FldGen)
\end{pmatrix}.
\]
\begin{samepage}
We have an isomorphism
\begin{equation}\label{eq:standard-solenoid-isomorphism}
 \Phi:\Omega_p^2=(\RR^2\times\QQ_p^2)/\zloverp^2\longrightarrow\Tt_p:
 \;
 [(x_\infty,x_p)]\longmapsto
 [(M_\infty x_\infty,M_px_p)].
\end{equation}
\end{samepage}

We use the following standard character criterion: if \(H\) is a subgroup
of a compact abelian group \(G\), then
\begin{equation}\label{eq:character-density-criterion}
 \overline H=G
 \quad\Longleftrightarrow\quad
 \set{\chi\in\widehat G:\restr{\chi}{H}=1}=\set{1}.
\end{equation}
Indeed, if \(\overline H\neq G\), the completeness of the characters of a
compact abelian group \cite[p.~621]{halmos}, applied to the nontrivial
quotient \(G/\overline H\), gives a nontrivial character of \(G\) which is
trivial on \(H\).

The standard description of the \(p\)-adic solenoid gives
\(\widehat{\Omega_p^2}\simeq\zloverp^2\); see
\cite[Ch.~II, \S1, pp.~509--511]{BerendAdelic}.  Recall that the
\(p\)-adic fractional part \(\{y\}_p\) is the unique element of
\(\zloverp\cap[0,1)\) such that \(y-\{y\}_p\in\ZZ_p\).  With the diagonal
convention used here, the character corresponding to
\(q\in\zloverp^2\) is
\begin{equation}\label{eq:standard-solenoid-character}
 \chi_q([(x_\infty,x_p)])
 =\exp\left(2\pi i\left(-\inner{q}{x_\infty}
                  +\{\inner{q}{x_p}\}_p\right)\right).
\end{equation}
This is well defined because \(\{r\}_p-r\in\ZZ\) for every
\(r\in\zloverp\).

For \(w\in\RR^2\), the restriction of \(\chi_q\) to the line generated by
\((w,0)\) is
\[
 \chi_q([(tw,0)])=\exp\left(-2\pi i t\inner{q}{w}\right),
 \qquad t\in\RR,
\]
and is therefore trivial precisely when \(\inner{q}{w}=0\).  Likewise, for
\(w\in\QQ_p^2\),
\[
 \chi_q([(0,sw)])=\exp\left(2\pi i\{s\inner{q}{w}\}_p\right),
 \qquad s\in\QQ_p.
\]
This restriction is trivial precisely when \(\inner{q}{w}=0\).  Indeed, if
\(a:=\inner{q}{w}\neq0\), then for \(s=p^{-1}a^{-1}\) its value is
\(e^{2\pi i/p}\neq1\).  Consequently, by
\eqref{eq:character-density-criterion}, the image in \(\Omega_p^2\) of a
line generated by \((w,0)\), or by \((0,w)\), is dense if and only if
\begin{equation}\label{eq:no-rational-orthogonal-vector}
 \inner{q}{w}\neq0
 \qquad\text{for every }q\in\zloverp^2\setminus\set{0}.
\end{equation}

Let \(f_1:=(1,0)^t\) and \(f_2:=(0,1)^t\).  A direct calculation using
\eqref{eq:standard-solenoid-isomorphism} shows that, for \(j=1,2\),
\begin{align}
 \Phi^{-1}\big(\set{\pi(t\e_j):t\in\RR}\big)
 &=\set{[(tM_\infty^{-1}f_j,0)]:t\in\RR}
 =\set{[(tw_j,0)]:t\in\RR},
 \label{eq:real-coordinate-preimage}\\
 \Phi^{-1}\big(\set{\pi(s\e_{j+2}):s\in\QQ_p}\big)
 &=\set{[(0,sM_p^{-1}f_j)]:s\in\QQ_p}
 =\set{[(0,sv_j)]:s\in\QQ_p},
 \label{eq:padic-coordinate-preimage}
\end{align}
where \(w_1:=(\sigma_2^\infty(\FldGen),-1)^t\),
\(w_2:=(-\sigma_1^\infty(\FldGen),1)^t\),
\(v_1:=(\sigma_2^p(\FldGen),-1)^t\), and
\(v_2:=(-\sigma_1^p(\FldGen),1)^t\).
Since \(\FldGen\notin\QQ\) and every embedding of \(K\) fixes \(\QQ\),
each of the numbers \(\sigma_j^\infty(\FldGen)\) and
\(\sigma_j^p(\FldGen)\) lies outside \(\QQ\).  Hence none of the four
vectors above is orthogonal to a nonzero vector in \(\zloverp^2\).  For
example, an equality
\[
 q_1\sigma_2^\infty(\FldGen)-q_2=0,
 \qquad q_1,q_2\in\zloverp,
\]
would imply either \(q_1=q_2=0\) or
\(\sigma_2^\infty(\FldGen)=q_2/q_1\in\QQ\).  The other three cases are
identical.  Equations \eqref{eq:real-coordinate-preimage} and
\eqref{eq:padic-coordinate-preimage}, together with
\eqref{eq:no-rational-orthogonal-vector}, therefore show that the two full
real lines and the two \(p\)-adic lines have dense image in \(\Tt_p\).

It remains to replace the full real lines by the positive rays in
\eqref{eq:realLines}.  The closure of such a ray is a closed subsemigroup of
the compact group \(\Tt_p\) and contains the identity.  Every such
subsemigroup is a group: if \(x\) belongs to it, choose a strictly increasing
sequence \(n_k\) such that \(n_kx\) converges.  Then
\((n_{k+1}-n_k)x\to0\), and hence
\((n_{k+1}-n_k-1)x\to-x\).  Thus the closure of the positive ray also
contains the negative ray, and consequently contains the full real line.
It is therefore equal to \(\Tt_p\), as required.
\end{proof}

\section{Minimal sets in $p$-Adic Tori}\label{sec: dense orbits}

In this section we work with the spaces 
$\Tt_p$ and $\Tpbar$ defined in \eqref{eq:Ttpdef} and
\eqref{eq:Tpdef} and the transformations $\theta$, $\zeta$ constructed in Lemma \ref{lem: NEW two stabilizers}.
Since the lattice $\Lambda_p$ is stabilised by $\theta$ and $\zeta$, these transformations act on $\Tt_p$.
Furthermore, since $V_p$ is invariant under $\SemiGroup$ action,
their action also descends to $\Tpbar$.
We study orbits of the semigroup $\SemiGroup$
acting on $\Tpbar$.
A crucial ingredient in our argument is 
a description of minimal subsets for this action.
A general result in this direction is due to Berend \cite{BerendAdelic}. However, our case is not directly covered by \cite{BerendAdelic}. We will therefore prove an analogue of \cite{BerendAdelic} in our setting, building on Berend's work:

\begin{theorem}\label{thm: densityNEW}
Any minimal closed $\SemiGroup$-invariant subset of $\Tpbar$ is finite.
\end{theorem}

The first step of the proof is:
\begin{lemma}\label{ref:BerendMagic}
If $B$ is a minimal closed $\SemiGroup$-invariant subset of $\Tpbar$, then $$B-B\neq \Tpbar.$$
\end{lemma}

\begin{proof} 
Using the terminology of \cite[beginning of \S 1 and \S 2]{BerendAdelic} the pair $(\Tt_p,\SemiGroup)$ is a \emph{group flow}, which means that  $\Tt_p$ is a compact metric abelian group and $\SemiGroup$ is a semigroup of continuous endomorphisms.  As noted in Section \ref{sec: periodic}, the space $\Tt_p$ is isomorphic to the standard solenoid $\Omega_p^2=(\RR^2\times \QQ_p^2)/\zloverp^2$, considered in \cite[Ch.~II]{BerendAdelic}. We now claim that $(\Tt_p,\SemiGroup)$ is also an $\mathcal F$-flow (see \cite[Def.~I.4]{BerendAdelic}). 
We recall that the definition of an $\mathcal F$-flow \cite[Def.~I.4]{BerendAdelic} requires the existence of a sequence of sub-semigroups $\SemiGroup:=\Sigma=\Sigma_1\supset \Sigma_2\supset\cdots$ each containing the identity element $I$, such that:
  \begin{enumerate}
      \item $(\Tt_p,\Sigma_k)$ is ergodic for every $k\geq 1$;
      \item $\Sigma/\Sigma_k$ has a finite group structure in the sense defined in \cite[Def.~I.4]{BerendAdelic};
      \item The collection of fixed points for $\Sigma_k$ becomes dense in $\Tt_p$ as $k\rightarrow \infty$.
  \end{enumerate}
According to \cite[Prop.~II.6]{BerendAdelic},  
to verify the $\mathcal F$-flow property,
it is enough to verify that the following two conditions hold: \begin{enumerate}
        \item[(i)] Some $g\in \SemiGroup$ acts ergodically on $\Tt_p$;
        \item[(ii)] There exists some prime $q$ not dividing $p$  such that $\det(g)$ is a $q$-adic unit for every $g\in \SemiGroup$.
    \end{enumerate}
        For part (i), we claim that $\theta$ acts ergodically on $\Tt_p$. Recall the definition of $\theta=\sigma(\epsilon_\infty^2)$  in Lemma \ref{lem: NEW two stabilizers}. Via the isomorphism between $\Tt_p$ and $\Omega_p^2$ mentioned above, the action of $\theta$ on $\Tt_p$ corresponds to the action of the matrix $\gamma:=\gamma(\epsilon_\infty^2)$ on $\Omega_p^2$, as defined in \eqref{eq:ChangeOfBasisMat} for the element $t=\epsilon_\infty^2$. As remarked in Remark \ref{rem:EVofChangeBasisMat}, the eigenvalues of $\gamma$ are $\epsilon_\infty^2,\epsilon_\infty^{-2}$, which are not roots of unity. This implies that $\gamma$ acts ergodically on $\Omega_p^2$, so $\theta$ acts ergodically on $\Tt_p$. Indeed, as the character group of $\Omega_p^2$ is $\zloverp^2$, this follows from \cite[Theorem 1]{halmos}, which states that $\gamma$ acts ergodically as long as its action on $\zloverp^2$ has no non-trivial finite orbits. But if $\gamma^nv=v$ for $0\neq v\in \zloverp^2$, then one of the eigenvalues of $\gamma$ must be a root of unity. 
  To verify part (ii), we also use  Remark \ref{rem:EVofChangeBasisMat}.
  Under the identification $\Tt_p\simeq\Omega_p^2$,
$\theta=\sigma(\epsilon_\infty^2)$ and $\zeta=\sigma(\tau_2^{-2})$ correspond to  
$\gamma(\epsilon_\infty^2)$ and $\gamma(\tau_2^{-2})$,
and their determinants can be computed in terms of real eigenvalues
of $\theta$ and $\zeta$, which allows to verify (ii).
This shows that $(\Tt_p,\SemiGroup)$ is an $\mathcal F$-flow.

Let $\Sigma_k$ denote the sequence of sub-semigroups which arise from the definition of $(\Tt_p,\SemiGroup)$ being an $\mathcal{F}$-flow.
Since $\Tpbar$ is a quotient of $\Tt_p$,
it is straightforward to check that 
with the same choice of $\Sigma_k$,
$(\Tpbar,\SemiGroup)$ satisfies conditions (a)--(c),
so that $(\Tpbar,\SemiGroup)$ is also an $\mathcal F$-flow.

     We can now use \cite[Thm.~I.2]{BerendAdelic} and the assumption that $B$ is a minimal closed $\SemiGroup$-invariant subset to conclude that $B$ is \emph{restricted} (see \cite[Def.~I.5]{BerendAdelic}). That is, if $B+C=\Tpbar$ for a closed $\SemiGroup$-invariant subset $C$, then $C=\Tpbar$. Clearly, if $B-B=\Tpbar$, then the fact that $B$ is restricted would imply that $-B=\Tpbar$, which further implies that $B=\Tpbar$. However, $\Tpbar$ always contains the orbit corresponding to the zero vector and therefore cannot be minimal, which violates the minimality of $B$.
\end{proof}

On the other hand, we will show:
\begin{lemma}\label{lem: density of difference}
	If $B$ is an infinite closed $\SemiGroup$-invariant subset of $\Tpbar$, then 
    $$
    B-B=\Tpbar.
    $$
\end{lemma}
Theorem \ref{thm: densityNEW} would now follow directly from Lemma \ref{ref:BerendMagic} and Lemma \ref{lem: density of difference}.

\subsection{Proof of Lemma \ref{lem: density of difference}}\label{sec:lem proof}

We use the basis $\e_1,\e_2,\e_3,\e_4$ defined in \eqref{eq:vector} and write $x\in \RR^2\times\QQ_p^2$ as
 $$
 x=[x]_{1}\e_1+[x]_{2}\e_2+[x]_{3}\e_3+[x]_{4}\e_4.
 $$
 Let $B$ be an infinite closed $\SemiGroup$-invariant subset of $\Tpbar$.
\begin{claim}\label{cl: oneCoordinatePt}
		There exist $i\in \set{1,2,4}$ and a scalar $[y]_i\in \RR\setminus\{0\}$, if $i=1,2$, and $[y]_4\in \ZZ_p\setminus\{0\}$ if $i=4$, such that $y:=[y]_i \e_i \in \pibar^{-1}(B-B)$.
	\end{claim}
\begin{proof}
Recall that $\barOurFunDo$ (see \eqref{eq:d0property1})
denotes a fundamental domain for $\Tpbar$.
Since $\Tpbar$ is compact and $B$ is infinite, 
there exist distinct $x_j',x'\in \barOurFunDo$
such that $\pibar(x_j')\rightarrow\pibar(x')$.
Setting $x_j=x_j'-x'$, we have
$$
x_j\in \pibar^{-1}(B-B), \quad x_j\to 0,\quad x_j\ne 0.
$$ 
Since $x_j'\in\barOurFunDo$, it follows that $[x_j]_3=0$ and $[x_j]_4\in\ZZ_p$.

    Assume first that $[x_{j_k}]_1\ne 0$ along a subsequence $j_k$. Note that the element $\theta$ expands the first coordinate and contracts the second coordinate and does not change the size of the fourth coordinate (cf. \eqref{eq:alpha def}). Fix $\eta>0$ and for each $k$ choose $\ell(k)$ such that $\eta\leq\left|[\theta^{\ell(k)}x_{j_k}]_1\right|\leq M\eta$, where $M>1$. This is possible with any $M$ larger than the operator norm of $\theta$. As $x_{j_k}\to 0$, and $\theta$ does not expand any other coordinate, any accumulation point $y$ of $\theta^{\ell(k)}x_{j_k}$ must be of the form $[y]_1\e_1$ for some non-zero real number $[y]_1$.
    When $[x_{j_k}]_2\ne 0$ along a subsequence $j_k$,
    the argument above works upon replacing $\theta$ by $\theta^{-1}$ and gives that $[y]_2\e_2\in \pibar^{-1}(B-B)$ with $[y]_2\ne 0$. In the remaining case,
    $x_j=[x_j]_4\e_4$ with $[x_j]_4\ne 0$ for all sufficiently large $j$.
\end{proof}

    We are now ready to conclude the proof of Lemma \ref{lem: density of difference} by establishing the following claim.
\begin{claim}\label{cl: density claim}
        Let $z=\pibar(y)$ with $y$ as in Claim \ref{cl: oneCoordinatePt}.
		Then
        $$
        \overline{\SemiGroup z}=\Tpbar.
        $$
\end{claim}

	\begin{proof}
		  Suppose that $y=[y]_4\e_4$ for some $[y]_4\in\ZZ_p\setminus\{0\}$.  
		Then 
		\begin{equation}\label{eq:4case}
			\SemiGroup z=\left\{\pibar\left(a^lb^{-k}p^{-2\ClassNumberNorm k}[y]_4 \e_4\right):l\in\ZZ, k\in \NN\right\},
		\end{equation}		
		where $a,b\in \ZZ_p^\times$ and $n_0\in \ZZ_{>0}$ are the parameters appearing in the definition of $\theta$ and $\zeta$ (cf. \eqref{eq:alpha def}--\eqref{eq:beta def}). 
        Since $a$ has infinite order, the closure
        $\overline{\{ a^{l} \mid l \in \ZZ \}}$
        contains a subgroup $1+p^r\ZZ_p$ for some $r \in \NN$. Therefore, 
        \begin{align*}
            z_k+S_k \subset \overline{\SemiGroup z}\quad
            \hbox{for all $k\in \NN$,}
        \end{align*}
        where $z_k:=\pibar(b^{-k}p^{-2n_0k}[y]_4\e_4)$
        and $S_k:=\pibar(p^{r-2n_0k}[y]_4\ZZ_p\e_4)$.
        It follows from Lemma \ref{lem:Irrationality}
        that $\cup_{k\in\NN} S_k$ is dense in 
        $\Tpbar$. In fact, using compactness for any neighborhood
        $U$ of identity in $\Tpbar$,  the sets $S_k$ are $U$-dense in $\Tpbar$ for all $k\ge k_0(U)$.
        Then also the sets $z_k+S_k$ are $U$-dense in $\Tpbar$ for all $k\ge k_0(U)$. Therefore, we conclude that 
        $\cup_{k\in\NN}  (z_k+S_k)$ is dense in 
        $\Tpbar$, which proves the claim.
        		
		Since the cases $i=1$ and $i=2$ are similar, we only consider the case $i=1$ here. In this case, 
        $$
            \SemiGroup z=\left\{\overline{\pi}(e^{l t_0}p^{- k n_0}e^{ k s_0}[y]_1\e_1):l\in \ZZ,k\in \NN\right\}
		$$
        with $[y]_1\in \RR\backslash\{0\}$.
        Since $p^{-n_0}e^{s_0}$ and $e^{t_0}$ are multiplicatively independent by Lemma \ref{lem: NEW two stabilizers}, we have 
        \begin{align*}
            \overline{\{e^{l t_0}p^{-k n_0}e^{k s_0}:\,l\in \ZZ,k\in \NN\}}=\RR_{>0}.
        \end{align*}
        Hence, the density follows from Lemma \ref{lem:Irrationality}.
	\end{proof}

\section{Proof of the main theorem: generic case}\label{sec:gen}

In this section, we prove the main theorem (Theorem \ref{thm: main}) for almost every $\alpha\in\RR$. Our key tool is the following proposition:

\begin{proposition}\label{prop: dense implies zeroV2}
Let $u_K\in \ZZ_p^\times$ be as in \eqref{eq: u def}
and $x\in X_{u_K}$ such that $\pdiag x$ is dense in $X_{u_K}$. 
Then for every $y\in Y$ in the fiber over $x$,
\begin{equation}\label{eq:need to prove}
   \inf\left\{0\neq N(u):u=\cvmatr{u_1}{u_2}{u_3}{u_4}\in A y,  u_3, u_4\in\ZZ_p\right\}=0.   
\end{equation}
\end{proposition}

In the proof we use the lattice $x_K\in X_{u_K}$,
introduced in \eqref{eq:THELattice},
and the elements $\theta,\zeta\in \Stab_A(x_K)$,
defined in \eqref{eq:alpha def}--\eqref{eq:beta def}.
Furthermore, 
we use the spaces 
$\Tt_p$ and $\Tpbar$ defined in \eqref{eq:Ttpdef} and \eqref{eq:Tpdef} and recall that 
$\piti:\Tt_p\rightarrow \Tpbar$
denotes the corresponding quotient map.

To simplify notation, we write
elements $y\in Y$ as pairs $(x;v)$, where $x\in X$ and $v\in  (\RR\times\QQ_p)^2/\Lambda_x$. Namely, the pair $(x;v)$ corresponds to the grid $\Lambda_x+v$. In particular, for $w\in\RR^2\times\QQ_p^2$ and any $(x;v)\in Y$, the notation $w\in (x;v)$ means that the vector $w$ belongs to the grid $v+\Lambda_x$. Similarly, $w\in x$ would mean that the vector $w$ belongs to the lattice $\Lambda_x$.

We embed $X$ in $Y$ via the natural injection $x\mapsto (x;0)$, where $0$ denotes the coset corresponding to the zero vector. Note that the restriction of the natural action of $G$ to $\GLT \ltimes (\RR\times\QQ_p)^2$ satisfies
$$
g(x;v)=(gx;gv),\quad  g\in G,\,(x;v) \in Y.
$$

Let $\gG=\gG_\infty\oplus \gG_p$ be the Lie algebra of $\GLT \ltimes (\RR\times \QQ_p)^2$. 
and 
\begin{equation*}
\begin{split}
    \uU_\theta^+:=\on{span}_\RR &\big\{v_\infty\in \gG_\infty: \on{Ad}_\theta(v_\infty,0)=\lambda (v_\infty,0)\text{,}\av{\lambda}>1\big\}\\&\oplus \on{span}_{\QQ_p}\big\{v_p\in \gG_p:\on{Ad}_\theta(0,v_p)=\lambda'(0,v_p)\text{,}\av{\lambda'}_p>1\big\},
    \end{split}
\end{equation*}
\begin{equation*}
\begin{split}
    \uU_\theta^-:=\on{span}_\RR &\big\{v_\infty\in \gG_\infty: \on{Ad}_\theta(v_\infty,0)=\lambda (v_\infty,0)\text{,}\av{\lambda}\leq1\big\}\\&\oplus \on{span}_{\QQ_p}\big\{v_p\in \gG_p:\on{Ad}_\theta(0,v_p)=\lambda'(0,v_p)\text{,}\av{\lambda'}_p\leq1\big\}.
    \end{split}
\end{equation*}
Since the adjoint action of $\theta$ on $\gG$
is diagonalizable, $\gG=\uU_\theta^+\oplus\uU_\theta^-$.
Let $U_\theta^+$ and $U_\theta^-$ denote the corresponding Lie subgroups. 
Then every element in a small neighborhood of the identity in 
$\GLT \ltimes (\RR\times \QQ_p)^2$
can be written uniquely as $u^+u^-$ with $u^+\in U_\theta^+$ and $u^- \in U_\theta^-$ 
close to the identity. Explicitly,
\[U^+_\theta=\set{u(z_1,z_2):z_1,z_2\in \RR},\]where
    \begin{equation}\label{eq:uz1z2def}
         u(z_1,z_2):=\pa{\left(\begin{pmatrix}
            1 & z_1\\
            0 & 1
        \end{pmatrix},I\right); \left(\colf{z_2}{0},\colf{0}{0}\right)},
    \end{equation}
    and $U^-_\theta$ is an open subgroup of 
    \begin{equation}\pa{
        \left(\begin{pmatrix}
            * & 0\\
            * & *
        \end{pmatrix},\on{SL}_2(\QQ_p)\right);\left(\colf{0}{*},\colf{*}{*}\right)}.
    \end{equation}
    
Using the continuity of the function $N$, to prove \eqref{eq:need to prove}, it suffices to prove the statement with the infimum is taken over vectors $u$ in the closure  $\overline{Ay}$. The following three lemmas   establish existence of elements of special form in $\overline{Ay}$. We will finally show that Proposition \ref{prop: dense implies zeroV2} follows from Lemma \ref{lem:Unipotent}.

\begin{lemma}\label{lem:PeriodicInClosure}
    There exists a sequence $a_n\in A$ such that $a_n (x;v)\rightarrow (\TheLattice;w)$, where 
    $w\in \Tt_p$ with $\SemiGroup \piti(w)$ finite. We may further ensure that 
    \begin{equation}\label{eq:a_n choice}
        a_n(x;v)=u_n^+u_n^-(\TheLattice;w) \textrm{ with } u_n^\pm\in U_\theta^\pm,\;\; u_n^\pm\to e,\;\; u_n^+\ne e. 
    \end{equation}
\end{lemma}
\begin{lemma}\label{lem:FixedInClosure}
    There exist a sequence $a_n\in A$ and $\ell\in \NN$ 
    such that $a_n(x;v)\rightarrow(\TheLattice;w)$
    where $\piti(w)$ is fixed by $\SubSemigroup$
    and \eqref{eq:a_n choice} is satisfied.
\end{lemma}
\begin{lemma}\label{lem:Unipotent}
    There exist $\ell\in \NN$, $w\in\Tt_p$ such that $\piti(w)$ is fixed by $\SubSemigroup$, and $0\neq z$ such that either $u(z,0)(\TheLattice;w)\in \Closure$ or  $u(0,z)(\TheLattice;w)\in \Closure$.
\end{lemma}
We now proceed to prove the above lemmas and eventually conclude the proof of Proposition \ref{prop: dense implies zeroV2}.
\begin{proof}[Proof of Lemma \ref{lem:PeriodicInClosure}]
   Using the density of $\pdiag x$ in $X_{u_K}$,
    we may choose a sequence $a_n'\in \pdiag$ such that,  
    $a_n'(x;v)\rightarrow (\TheLattice;w')$ in $Y$ for some $w'\in \Tt_p$. We may  write $a_n'(x;v)=u(z_n,z_n')v_n(\TheLattice;w')$, where 
    $z_n,z_n'\to 0$ and $v_n\in U_\theta^-$, $v_n\to e$.
    Furthermore, using the density of $\pdiag x$ in $X_{u_K}$, we may ensure that $z_n\neq 0$.
        
    We choose $w\in \overline {\SemiGroup w'}$ such that  $\overline {\SemiGroup \piti(w)}$ is a minimal closed $\SemiGroup$-invariant subset in $\Tpbar$. Then 
    there exists a sequence $b_n\in \SemiGroup$ such that $b_n (\TheLattice;w')=(\TheLattice;w+w_n)$ with $w_n\rightarrow 0$. We take a subsequence $j_n$ such  that 
    $$
    b_n u(z_{j_n},z_{j_n}')v_{j_n}b_n^{-1} \to e.
    $$
    Then
$$b_na_{j_n}'(x;v)=b_n u(z_{j_n},z_{j_n}')v_{j_n}b_n^{-1} (\TheLattice;w+w_n)\rightarrow (\TheLattice;w),$$ and the sequence $a_n=b_na_{j_n}'$ satisfies \eqref{eq:a_n choice}. Here we crucially use that $z_{j_n}$ is always non-zero and that conjugation by any element of $A$ preserves $U_\theta^{\pm}$, and amounts to multiplying $z_{j_n}$ by a non-zero number. The orbit $\SemiGroup \piti(w)$ is finite 
by Theorem \ref{thm: densityNEW}.
\end{proof}

\begin{proof}[Proof of Lemma \ref{lem:FixedInClosure}]
Let $a_n', w'$ satisfy Lemma \ref{lem:PeriodicInClosure}. As  $\SemiGroup \piti(w')\subset \Tpbar$ is finite, by Pigeonhole Principle, there exist $i<j\in \NN$ with $\zeta^i \piti(w')= \zeta^j\piti(w')$. Define $w=\zeta^i w'$ and note that $\zeta^i a_n'(x;v)\rightarrow(\TheLattice;w)$ as $n\rightarrow\infty$.  Consider $H:=\set{g\in \SemiGroup:\piti(gw)=\piti( w)}$ and note that we have just shown that $\zeta^{j-i}\in H$. As before, there exist $m<n\in \NN$ with $\theta^m \piti(w)= \theta^n \piti(w)$. By the claim above applied to $\theta^{-m}$ we have that $\piti(w)=\piti( \theta^{n-m}w)$, so $\theta^{n-m}\in H$. Setting $a_n=\zeta^ia_n'$ and $\ell={\rm lcm}\pa{j-i,n-m}$ we see that $\piti(w)$ is fixed by $\SubSemigroup$, as required. Moreover, it is straightforward to see that \eqref{eq:a_n choice} holds again.
\end{proof}
\begin{proof}[Proof of Lemma \ref{lem:Unipotent}]
Set $y=(\TheLattice;w')$ where $w'$ satisfies the statement of Lemma \ref{lem:FixedInClosure}.
    Let $y_n= a_n(x;v)$ for $a_n\in A$ such that $y_n\rightarrow y$ as $n\rightarrow \infty$. 
	For $n$ large enough, we can write
$y_n=u_n^+u_n^- y$,
	with $u_n^+\in U^+_{\theta}$ and $u_n^- \in U^-_\theta$ such that $u_n^+,u_n^- \rightarrow e$ and $u_n^+\ne e$. Recall that the action by conjugation by $\theta$ is uniformly expanding on $U_\theta^+$ and neutral/contracting on $U_\theta^-$, so that  
	we can choose $k_n\in \NN$ minimal such that the maximal coordinate of $\theta^{\ell k_n}u_n^+\theta^{-\ell k_n}$ is in the interval $[2,M]$ with $M>2$ is a fixed constant depending only on the expansion rate of $\theta^\ell$ on $U^+_\theta$. Perhaps redefining $u_n^+$ and $k_n$ to be appropriate subsequences,  without loss of generality, we may assume that 
    $\theta^{\ell k_n}u_n^+ \theta^{-\ell k_n}\rightarrow u^+=u(z,z')$, where $2\leq \max\{\av{z},\av{z'}\}\leq M$. Also $\theta^{\ell k_n}u_n^- \theta^{-\ell k_n}\rightarrow e$. 

As $\Tt_p$ is compact, by passing to a subsequence and redefining $k_n$, we may assume that $\theta^{\ell k_n}w'\to \tilde w$ as $n\to \infty$ for some $\tilde{w}$. Note that $\piti(\tilde w)$ must also be $\SubSemigroup$-fixed in $\Tpbar$. Putting everything together, we get the convergence	\begin{equation}
		\theta^{\ell k_n}y_n=\theta^{\ell k_n}u_n^+\theta^{-\ell k_n}\theta^{\ell k_n}u_n^- \theta^{-\ell k_n}(\theta^{\ell k_n}y)\rightarrow u^+(\TheLattice;\tilde w).
	\end{equation}
	  Recall that $u^+=u(z,z')$, so we have $u(z,z')(\TheLattice;\tilde w)\in \Closure$. If $z=0$, the lemma is proved with $u(0,z')$ and  with $w=\tilde w$. Assume now that $z\neq 0$.  One may directly calculate for any $m,n\in \NN$ that
\begin{equation}
    \theta^{\ell m}\zeta^{\ell n} u(z,z') (\theta^{\ell m}\zeta^{\ell n})^{-1}=u(e^{2\ell mt_0+2\ell ns_0}z,e^{\ell mt_0+\ell ns_0}p^{-\ClassNumberNorm\ell n}z').
\end{equation}
Using the multiplicative independence of $e^{2s_0}$ and $e^{2t_0}$ established in Lemma \ref{lem: NEW two stabilizers}, we find a subsequence of integers $|m_{k}|,n_{k}\rightarrow\infty$ such that $m_{k}t_0+n_{k}s_0\rightarrow 0$ as $k\rightarrow \infty$. Consequently,
\[(\theta^{\ell m_{k}}\zeta^{\ell n_{k}})u(z,z')(\theta^{\ell m_{k}}\zeta^{\ell n_{k}})^{-1}\to u(z,0).\]
As before, we redefine $m_{k},n_{k}$ to be subsequences with $\theta^{\ell m_{k}}\zeta^{\ell n_{k}}\tilde w$ converging to $w\in \Tt_p$ and note that $\piti(w)$ is also $\SubSemigroup$-fixed. We have the convergence
\begin{align*}
\theta^{\ell m_{k}}\zeta^{\ell n_{k}}u(z,z')(\TheLattice;\tilde w)&=\\    
 \theta^{\ell m_{k}}\zeta^{\ell n_{k}}u(z,z')\pa{\theta^{\ell m_{k}}\zeta^{\ell n_{k}}}^{-1}\theta^{\ell m_{k}}\zeta^{\ell n_{k}}(\TheLattice;\tilde w)&\to u(z,0)(\TheLattice;w)\in \Closure,
\end{align*}
as required.
\end{proof}
\begin{proof}[Proof of Proposition \ref{prop: dense implies zeroV2}]
To prove Proposition \ref{prop: dense implies zeroV2}, using continuity of the function $N$, it is enough to show the existence of vectors $u$ in grids belonging to the whole orbit $\Closure$ with $N(u)$ arbitrarily small but non-zero. Note that for any $w_1,w_2\in\OurFunDo$ such that $\pibar(w_1)=\pibar(w_2)$, then $N(w_1)=N(w_2)$. Let $\ell\in \NN$ and $\SubSemigroup$-fixed $\piti(w)\in \Tpbar$ be satisfying the statement of Lemma \ref{lem:Unipotent}.

First consider the case  $u(z,0)(\TheLattice;w)\in \Closure$ with $z\neq 0$. We claim that in this case, the proposition would follow from proving the following statement: There exist $x_1,x_2\in \RR$ and $x_4\in \ZZ_p$ with $x_1(zx_2)<0$ and $x_4\neq 0$ such that for every $m_1\in \ZZ,m_2\in \NN$, there exists $x_3\in \ZZ_p$ possibly depending on $m_1$ and  $m_2$, such that the vector
\begin{equation}
    \label{eq:SmallVecInClosure}
\pa{ \binom{x_1+e^{2\ell m_1t_0+2\ell m_2s_0}z x_2}{x_2}, \binom{x_3}{x_4}}\in(\theta^{\ell m_1}\zeta^{\ell m_2}) u(z,0)(\TheLattice;w)\in \Closure.
\end{equation}
Indeed, if such vectors exist, we have 
$$
N\pa{ \binom{x_1+e^{2\ell m_1t_0+2\ell m_2s_0}z x_2}{x_2}, \binom{x_3}{x_4}}=\av{x_1+e^{2\ell m_1t_0+2\ell m_2s_0}z x_2}\av{x_2}\av{x_4}_p.
$$
Since $x_2$ and $x_4$ do not depend on $m_1$ and $m_2$, it is enough to show that $\av{x_1+e^{2\ell m_1t_0+2\ell m_2s_0}z x_2}$ can be arbitrarily small. To prove this, we invoke Lemma \ref{lem: NEW two stabilizers} which asserts that $e^{2 t_0}$ and $e^{2 s_0}$ are multiplicatively independent. As a result, \begin{equation}\label{eq:multindep}
   \overline{\{e^{2\ell m_1t_0+2\ell m_2s_0}:m_1\in\ZZ,m_2\geq 0\}}=\RR_{>0}.\end{equation}
   Since $x_1$ and $zx_2$ have opposite signs, the proposition follows once we show \eqref{eq:SmallVecInClosure}.

To this end, note first that for $m_1\in \ZZ$ and $m_2\in \NN$,
   \begin{equation}
      (\theta^{\ell m_1}\zeta^{\ell m_2}) u(z,0)(\TheLattice;w)= u(e^{2\ell m_1t_0+2\ell m_2s_0}z,0)(\theta^{\ell m_1}\zeta^{\ell m_2})(\TheLattice;w)\in \overline{A(x;v)}.
   \end{equation}
A quick calculation shows that it is enough to find $x_1,x_2\in \RR,\,x_4\in \ZZ_p$, with $x_1(zx_2)<0$ and $x_4\neq 0$ such that 
$\big( \binom{x_1}{x_2}, \binom{x_3}{x_4} \big)\in (\theta^{\ell m_1}\zeta^{\ell m_2})(\TheLattice;w)$ 
for some $x_3\in \ZZ_p$. But since $\piti(w)$ is $\SubSemigroup$-fixed, each vector $v\in (\theta^{\ell m_1}\zeta^{\ell m_2})(\TheLattice;w)$ is of the form $v=v_1+\big( \binom{0}{0}, \binom{u}{0} \big)$ where $v_1\in (\TheLattice;w)$ and $u\in\ZZ_p$. It follows that as $x_3$ is allowed to depend on $m_1,m_2$, it is enough to find a vector $\big( \binom{x_1}{x_2}, \binom{x_3'}{x_4} \big)\in (\TheLattice;w)$ with $x_1,x_2\in\RR,x_3',x_4\in\ZZ_p$ satisfying $x_1(zx_2)<0,x_4\neq 0$. Let  $w_0\in (\TheLattice;w)\cap \Dd_0$ be a representative. If $w_0$ itself does not satisfy all the above requirements, let $\bar x_K$ denote the corresponding real lattice. We therefore find $\binom{\theta_1}{\theta_2}\in \bar x_K$ such that the condition $\theta_1(z\theta_2)<0$ holds. We lift it to a vector $\Theta=\big( \binom {\theta_1}{\theta_2}, \binom{\theta_3}{\theta_4} \big)\in \TheLattice$ with $\theta_4\neq 0$. Finally, we choose $n\in \NN$ large enough such that the vector
$$
\pa{ \binom{x_1}{x_2}, \binom{x_3'}{x_4} }:=w_0+p^n\Theta\in (\TheLattice;w)
$$
has $x_1(zx_2)<0$, $x_3',x_4\in \ZZ_p$ and $x_4\neq 0$.
This establishes the proposition in the first case.

Now we deal with the case when  $\piti(w)\in \Tpbar$ is fixed by
$\SubSemigroup$, $u(0,z)(\TheLattice;w)\in \Closure$ and $z\neq 0$. Analogous calculation as in the previous case shows that the grid $\theta^{\ell m_1}\zeta^{\ell m_2}u(0,z)(\TheLattice;w)$ contains vectors of type
\begin{equation}\label{eq:vec small 2}
       \pa{\colf{x_1+e^{\ell m_1t_0+\ell m_2s_0}p^{-n_0\ell m_2}z}{x_2},\colf{x_3}{x_4}}
   \end{equation}
with $x\in \theta^{\ell m_1}\zeta^{\ell m_2}(\TheLattice;w)$.   
Using that $e^{\ell t_0}$ and $e^{\ell s_0}p^{-n_0\ell}$ are multiplicatively independent by Lemma \ref{lem: NEW two stabilizers},  following our argument in the previous case, it is enough to show that $(\TheLattice;w)$ contains a vector $\pa{ \binom{x_1}{x_2}, \binom{x_3'}{x_4}}$ with $x_1 z<0$, $x_4\neq 0$ and $x_3',x_4\in\ZZ_p$. The existence of such a vector follows from a minor modification of our earlier argument.
\end{proof}

The unboundedness of the semigroup $\pdiag$ gives the following.
\begin{lemma}\label{lem: mixing in X_u}
    For Lebesgue almost every $v\in \RR$, any $u\in \ZZ_p^\times$, the $\pdiag$-orbit of the coset 
    \begin{equation}
		\left(\begin{pmatrix}
			1 & 0\\
			v & 1
		\end{pmatrix},\begin{pmatrix}
			u & 0\\
			0 & 1
		\end{pmatrix}\right)\GLZP 
	\end{equation}
    is dense in $X_u$.
\end{lemma}
\begin{proof}
For any $u\in \ZZ_p^\times$, the one-parameter subgroup $\{a_t:t\in\RR\}\subset A$, where
$$
a_t:=\left(\begin{pmatrix}e^t&0\\0&e^{-t}\end{pmatrix},I_2\right)\in A,
$$
preserves $X_u$. The space $X_u$ is a quotient of $\SL_2(\RR\times\QQ_p)$ by a lattice, and it carries the induced $\SL_2(\RR\times\QQ_p)$-invariant probability measure $\mu_u$.

Fix $s\in\RR\setminus\{0\}$. The stable horospherical subgroup of $a_s$ is
$$
G_{a_s}^-:=\left\{g\in \SL_2(\RR\times\QQ_p): a_s^n g a_s^{-n}\stackrel{n\to+\infty}{\longrightarrow} e \right\}
=
\left\{\left(\begin{pmatrix}1&0\\ v&1\end{pmatrix},I_2\right)\colon v\in\RR\right\}.
$$
Since $\{a_t:t\in\RR\}$ is unbounded, by Mautner's property (see \cite[\S 2.1]{margulis1996measure} for the $S$-arithmetic setting), $a_s$ acts ergodically on $X_u$.
By the standard Hopf argument, for Lebesgue almost every point on the horospherical leaf
$$
\left\{\left(\begin{pmatrix}1&0\\ v&1\end{pmatrix},\begin{pmatrix}u&0\\0&1\end{pmatrix}\right)\GLZZP  : v\in\RR\right\},
$$
the backwards $a_s$-orbit is equidistributed in $X_u$. Since $a_s\in A$, the $A$-orbit of each such point is also (in particular) dense in $X_u$.
\end{proof}

\begin{proof}[Proof of Theorem \ref{thm: main} (the generic case)]
	By Lemma \ref{lem: mixing in X_u}, for almost every $\alpha\in \RR$ and for every $u\in \ZZ_p^\times$, the lattice
	\begin{equation}\label{eq:Lgammadef}
		\Delta_{\alpha,u}:=\left(\begin{pmatrix}
			1 & 0\\
			1/\alpha & 1
		\end{pmatrix},\begin{pmatrix}
			u & 0\\
			0 & 1
		\end{pmatrix}\right)\GLZZP 
	\end{equation}
	has a dense $\pdiag$-orbit in $X_u$.
    In particular, for $u=u_K$ as in \eqref{eq: u def}, we have that $\pdiag\Delta_{\alpha,u_K}$ is dense in $X_{u_K}$.
	Proposition \ref{prop: dense implies zeroV2} implies that for every $v\in \RR^2\times \QQ_p^2$:
	\[\inf\set{N(u)\neq 0:u=\cvmatr{u_1}{u_2}{u_3}{u_4}\in A(\Delta_{\alpha,u_K};v),  u_4\in\ZZ_p}=0.\]
    In particular, this holds for $v=\left(\colf{0}{\delta/\alpha},\colf{0}{\kappa}\right)$,
so that the conditions of Corollary \ref{lem: dani} hold for $\overline{v}:=(1/\alpha,\delta/\alpha,\kappa)\in\RR^2\times\ZZ_p$.
We conclude that
$$
\liminf_{\av{q}\rightarrow\infty}|q|\inn{q\alpha+\delta}\abs{q+\kappa}_p=0
$$
for almost every $\alpha$ and all $\delta,\kappa$, as needed.

\end{proof}

\section{Proof of the main theorem: quadratic irrational case}\label{sec: quad}

Let $\alpha\in (0,1)$ be a quadratic irrational. Throughout, let $\Delta_{\alpha,u_K}$ be as in \eqref{eq:Lgammadef}, where $u_K$ is defined in \eqref{eq: u def}.
Arguing as in the previous section, 
in order to prove Theorem \ref{thm: main},
it would be enough to show that $\pdiag\Delta_{\alpha,u_K}$ is dense in $X_{u_K}$. We first want to reduce to a density statement in $X_1$ instead of $X_{u_K}$.

\begin{lemma}\label{lem:Xu-to-X1}
For every $u\in \ZZ_p^\times$, left translation by
$
d_u:=(I,\diag(1,u))
$
induces an $\pdiag$-equivariant homeomorphism
$
\iota_u:X_1\longrightarrow X_u.
$
In particular, using the notation \eqref{eq:Lgammadef}, we have that $\pdiag\Delta_{\alpha,u}$ is dense in $X_u$  if and only if $\pdiag\Delta_{\alpha,1}$ is dense in $X_1$. 
\end{lemma}

\begin{proof}
By definition,
$$
X_u=(I,\diag(1,u))\SL_2(\RR\times\QQ_p)\GLZP 
=d_uX_1.
$$
Hence left translation by $d_u$ gives a homeomorphism $X_1\to X_u$. Since $\pdiag$ commutes with $d_u$, this map is $\pdiag$-equivariant. Since $\iota_u(\pdiag\Delta_{\alpha,1})=\pdiag\Delta_{\alpha,u}$ the lemma follows.
\end{proof}

Let $P=\overline{\pdiag\Delta_{\alpha,1}}$ denote the closure in $X_1$. Using Lemma \ref{lem:Xu-to-X1} we need to show that  $P=X_1$. To this end we will use \cite{Aka-Shapira} to show that $P$ contains a sequence of equidistributing geodesics. 
Note that 
$$X_1=\SL_2(\RR\times\QQ_p)\GLZP \cong \SL_2(\RR\times\QQ_p)/\SL_2(\ZZ[\tfrac1p]).
$$
The setting of \cite{Aka-Shapira} is similar, namely $\operatorname{PGL}_2(\RR\times\QQ_p)/\operatorname{PGL}_2(\ZZ[\tfrac1p])$, which is slightly more general than we need here. Before giving further details, we record first the main matrix computation of this argument.

Set
$$
A_\infty
:=
\left\{
\left(\diag(e^t,e^{-t}),I\right): t\in\RR
\right\}
\subseteq \pdiag,\, b_n
:=
\left(p^{-n}I_2,\diag(1,p^{-2n})\right)\in \pdiag.
$$

\begin{lemma}
For every $n\in \NN$ we have
$
 \Delta_{p^{2n}\alpha,1}\in P$.
\end{lemma}

\begin{proof}
Noting that
$
\delta:=\left(\diag(1,p^{2n}),\diag(1,p^{2n})\right)\in \GLZP ,
$
and that $A_\infty\subseteq \pdiag$, the lemma follows from
the following computation:
\begin{align*}
b_n\Delta_{\alpha,1}
&=
\left(p^{-n}I_2,\diag(1,p^{-2n})\right)
\left(
\begin{pmatrix}
1 & 0\\
\alpha^{-1} & 1
\end{pmatrix},
I_2
\right)\delta\GLZP \\
&=
\left(
\begin{pmatrix}
p^{-n} & 0\\
p^{-n}\alpha^{-1} & p^{n}
\end{pmatrix},
I_2
\right)\GLZP \\
&=
\left(\diag(p^{-n},p^{n}),I_2\right)
\left(
\begin{pmatrix}
1 & 0\\
p^{-2n}\alpha^{-1} & 1
\end{pmatrix},
I_2
\right)\GLZP \\
&=
a_{-n\log p}\,\Delta_{p^{2n}\alpha,1}.
\end{align*}
\end{proof}




\begin{proof}[Proof of Theorem \ref{thm: main} (the quadratic irrational case)]
    As explained after Lemma \ref{lem:Xu-to-X1}, it is enough to show that $P=X_1$. Let $S=\set{\infty,p}$. In the terminology of \cite{Aka-Shapira}, the forward orbit under $A_\infty$ of the point $\Delta_{p^{2n}\alpha,1}$ converges to an $S$-adic lift of the closed geodesics corresponding to $p^{2n}\alpha$. For $n\in \NN$, these orbits lie on a rational branch in the
$S$-Hecke graph. Indeed, by \cite[Definition~4.3]{Aka-Shapira}, such a branch is rational ($\omega$ in their notation is simply $I_2$).
By \cite[Theorem~4.8(4)]{Aka-Shapira} this branch is non-degenerate. Hence
\cite[Theorem~4.8(1)]{Aka-Shapira} implies that the associated 
$A_\infty$-orbits equidistribute in $X_1$. As these orbits lie in the closure $P$, this shows that $P=X_1$, as needed. 
\end{proof}

\bibliographystyle{plain}
\bibliography{BibErg}{}
\bigskip
\noindent
\textsc{Menny Aka, Department of Mathematics, ETH Z\"urich, 8092
Z\"urich, Switzerland}

\noindent
\textit{Email address:}
\texttt{menny.akka@math.ethz.ch}

\bigskip
\noindent
\textsc{Alex Gorodnik, Institut f\"ur Mathematik, University of Z\"urich, 8057 Z\"urich, Switzerland.}

\noindent
\textit{Email address:}
\texttt{alexander.gorodnik@math.uzh.ch}

\bigskip
\noindent
\textsc{Pankaj Vishe, Department of Mathematical Sciences, Durham University, 
Durham DH1 3LE, U.K. (corresponding author)}

\noindent
\textit{Email address:}
\texttt{pankaj.vishe@durham.ac.uk}

\bigskip
\noindent
\textsc{Yuval Yifrach, Institut f\"ur Mathematik, University of Z\"urich, 8057 Z\"urich, Switzerland.}

\noindent
\textit{Email address:}
\texttt{yuval.yifrach@math.uzh.ch}

\end{document}